\documentclass[12pt]{article}
\newcommand{\nocopyright}{
No Copyright\thanks{
The authors hereby waive all copyright
and related or neighboring rights to this work,
and dedicate it to the public domain.
This applies worldwide.
}}
\title{Blowing bubbles on the torus}
\author{Peter Doyle \and Jean Steiner}
\date{Original version dated 1 October 2005\\ Addendum dated 14 April 2009
\\ \nocopyright}

\usepackage{amsmath,amscd,amssymb, amsthm, color, graphics,  color,epsfig,bm,epsf,times}

\newcommand{\mathsym}[1]{{}}

\newcommand{\note}[1]{}
\theoremstyle{plain}
\newtheorem{theorem}{Theorem}

\newtheorem{prop}{Proposition}
\newtheorem{cor}{Corollary}
\theoremstyle{definition}
\newtheorem{definition}{Definition}
\newtheorem{convention}{Convention}
\theoremstyle{remark}
\newtheorem{rmk}{Remark}

\newcommand{\dif}{\mathrm{d}}

\newcommand{\etwophi}{e^{2\phi}}

\def\bi{\begin{itemize}}
\def\ei{\end{itemize}}
\def\ben{\begin{enumerate}}
\def\een{\end{enumerate}}
\def\be{\begin{equation*}}
\def\ee{\end{equation*}}
\def\bea{\begin{eqnarray*}}
\def\eea{\end{eqnarray*}}
\def\ie{{\it i.e.\ }}

\begin{document}

\maketitle
\begin{abstract}
We consider the regularized trace 
of the inverse of the Laplacian on a skinny torus.
With its flat metric, a skinny torus has large trace,
but we show that there are conformally equivalent metrics
making the trace close to that of a sphere of the same area.
This behavior is in sharp contrast to that of the log-determinant,
a well-known spectral invariant which is extremized at
the flat metric on any torus.
Our examples are bubbled tori, where you take a sphere, discard polar regions,
and glue top to bottom.
In a addendum,
we belatedly notice that our bubbled tori
have trace \emph{less than} the sphere,
and outline how to exploit this to get
Okikiolu's result that
by means of a conformal factor depending only on longitude,
any torus can be made to have trace less than the sphere.

\vspace{0.5cm}
{\bf AMS 2000 Mathematics Subject Classification }\quad
58J65, 60J05, 58J50, 53A30
\end{abstract}

\section{Introduction}\label{sec.intro}
On a surface, the Laplacian operator, $\Delta$  is a linear operator with discrete eigenvalues $\{\lambda_j\}_{j=0}^\infty$.   As any good linear algebra student might hope, it turns out that the determinant and the trace are interesting spectral invariants, which describe geometric features of the underlying surface.   The results in this paper concern the behavior of regularized trace on surfaces of genus $1$, however the regularized trace is closely connected to the determinant, so in order to contextualize the present  work, we briefly recall some relevant highlights of the rich and varied story for the determinant.


The log-determinant of the Laplacian is a spectral invariant that is defined through a zeta regularization process.  Heuristically, the determinant is defined by  $ \det(\Delta))=\exp[-\sum \lambda_j ]= \exp[-Z'(0)]$, where $Z(s)=\sum \lambda_j^{-s}$ is the zeta function and a meromorphic continuation is a required to give a  rigorous  definition for the determinant.   The study of the log-determinant has proven to be a rich area for the Laplacian on surfaces and in many other contexts. To hint at reach of the log-determinant on surfaces, we recall some  of the results in the oft-cited work of Osgood, Phillips, and Sarnak in \cite{Ops}, in which connections are given between this spectral invariant  and results concerning the uniformization theorem on surfaces,  the re-normalized Ricci flow,  critical Sobolev inequalities, and  number theory.    The extremal properties of the log-determinant demonstrate that this functional `hears' the constant curvature metric within the conformal class of the metric.  The following theorem descibes the extremal behavior within conformal classes for any compact surface without boundary, and it also describes the extremal behavior across conformal classes for surfaces of genus $1$.
\begin{theorem}[Osgood-Phillips-Sarnak \cite{Ops}]\label{thm.ops} On $(M,g)$, among all metrics conformally equivalent to $g$  with a  fixed area, $\det(\Delta)$  is maximized at the constant curvature metric.\\
Among all tori, $(T^2, g)$ with  fixed area, $\det(\Delta)$ is maximized at the (properly scaled) torus corresponding to the hexagonal lattice $\mathbb{R}^2/\Lambda$, where $\Lambda=<(1, 0), (\frac{1}{2} , \frac{\sqrt{3}}{2})>$.
\end{theorem}
On surfaces of genus one or higher, the extremization of the log-determinant is a straightforward consequence of understanding how $\det(\Delta)$ changes under a conformal change of metric.  The behavior of $\det(\Delta)$ is described by the Ray-Singer-Polyakov formula, which was considered by Polyakov  in \cite{Polyakov1} and \cite{Polyakov2}, and by Ray and Singer in \cite{Raysinger}.   

Heuristically, the trace is given by  $`Tr(\Delta^{-1})=\sum \frac{1}{\lambda_j}'$, but of course this sum does not converge, so a regularization is required in order to make a useful definition.  We will work with the $\tilde{Z}(1)$ (defined subsequently in Equation \ref{e.defregzeta}), where the regularization is made {\it via } the spectral zeta function defined in \eqref{e.defzeta}.   The regularized trace can also be  by considering a mass-like term in the Green's function for the Laplacian, and the comparison of these two regularizations is noted in Proposition \ref{p.massisspecdensity}, and the statement   appears in Morpurgo's work in \cite{Morpurgo_Duke} and in the work of Steiner(?the second author?) in \cite{steiner_duke}.    In contrast to the story of the log-determinant, the story for the regularized trace, $\tilde{Z}(1)$, is quite incomplete.  Part of the motivation for the present work comes from a result of Okikiolu in \cite{Okikiolu}, which shows that the density for the regularized trace emerges as a condition for criticality of the log-determinant in a more general setting, and a related result appears in work of Richardson in \cite{Richardson}.   The regularized trace was considered by Morpurgo in \cite{Morpurgo}, where a formula is given for the behavior of the regularized trace under a conformal change of metric, and he also considers its extremal behavior on the sphere.  Among metrics on the sphere, in \cite{Morpurgo}, Morpurgo showed that $\tilde{Z}(1)$ is minimized at the standard round metric by observing that the functional arising in the Polyakov-type formula given below in \eqref{e.changeofzeta} is the same functional that arises in a logarithmic-Hardy-Littlewood-Sobolev inequality proven by Beckner in \cite{Beckner} and Carlen and Loss in \cite{Carlenloss}.   There is a nice connection between the extremal behavior of the trace and the determinant on the sphere, since It turns out that the inequality which shows that the trace is minimized   is actually dual to the Onofri-Beckner inequality that gives  the maximization  of the log-determinant on the sphere.   Additionally, in \cite{steiner_duke},  Steiner shows that on the sphere, the point-wise varying behavior of the density for the regularized trace is  exactly captured by the potential for scalar curvature.    So, the story of the regularized trace  and its density on the sphere is reasonably well-understood, however much less is known for surfaces of genus $1$.   On the space of tori with flat-metrics, the zeta function satisfies a functional equation, and Chiu, in \cite{Chiu} showed that  the log-determinant and the regularized trace are linearly related, as described below in \eqref{e.tracedet}, so we can describe the behavior across conformal classes in terms of the behavior of the log-determinant.    In this paper, we work with rectangular `long-skinny' tori, for which the first non-zero eigenvalue will be suitably small, and we consider  the behavior of the regularized trace  within the conformal class of the flat metric.    
\begin{theorem}
For a  torus $(T^2, g_{T^2})$, with flat metric $g$  and first non-zero eigenvalue $\lambda_1<8\pi$ (hence `long-skinny'), among the metrics conformal to $g$, the regularized trace $\tilde{Z}(1)$ is not minimized at the constant curvature metric.  In fact, it is possible to exhibit a conformal change of metric, $g_\phi=e^{2\phi} g_{T^2}$, for which the value of the regularized trace approaches that of the standard round sphere, \ie $\tilde{Z}_{\phi}(1)\approx \tilde{Z}_{S^2}(1)$.
\end{theorem}
The conformal change of metric which provides the key to proving the theorem is obtained by `blowing a bubble'  in the torus.  The following figure gives a characature of the conformal factor which conformally maps  most of the torus to a sphere where the top and bottom caps have been removed and the top and bottom rims have been identified.  
\begin{figure}[h]\label{fig.torus}
\hspace{1.5in}\includegraphics[width=2in]{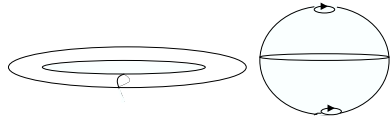}
\caption{Skinny torus has $\tilde{Z}_{T^2_a}(1)=O(a)$.  Bubbled torus has $\tilde{Z}(1)\approx \tilde{Z}_{S^2}(1) \approx -.189$ }
 \end{figure}

\paragraph{Organization of paper.}
In Section \ref{sec.background}, we define the regularized trace, and we give the precise connection between the regularized trace and the Green's function for the Laplacian.    We also describe how the regularized trace and its density change under conformal changes of metric.  Additionally, we give the relation between the regularized trace and the log-determinant for flat-tori, which characterizes the behavior across conformal classes.  In Section \ref{sec.longskinny}, we discuss the known local extremization of the regularized trace, which motivates our definition of long-skinny tori.  Finally, the main result of Theorem \ref{thm.blowbubble}  appears in Section \ref{sec.blowbubbles}, where we give an explicit conformal factor, and explain how the regularized trace was computed.  We also outline an alternate method of proof, which involves making estimates on Green's functions.

\section{Background}\label{sec.background}
 Throughout this paper, $(T^2,g_{T^2})$ will denote a surface of genus $1$, where $g_{T^2}$ is the flat metric with area $1$, and $T^2$ is obtained  by normalizing a quotient of  $\mathbb{R}^2$ by the lattice $\Lambda=<(1,0), \tau>$ for  some $\tau\in \mathbb{R}^2$.  Our normalization requires that the quotient has co-area $1$.    We will also be working with non-uniform metrics $g$ on $T^2$, and for any  
Riemannian metric tensor $g$, we use $\nabla$, $\Delta=-div\cdot \nabla$, and $\dif V$
to denote the  covariant derivative, the positive Laplacian operator,
and the volume form, respectively.  

\paragraph{The Green's function for the Laplacian} Since the regularized trace may be defined by way of the Green's function, it is an  object of central interest in this paper.  The  {\it Green's function} for the Laplacian operator is the integral kernel for the operator $\Delta^{-1}$.  
The Laplacian operator, $\Delta$,  is a  linear operator on the Sobolev space $H_2(M)\subset L^2(M,g)$,  and $\Delta$ has a one  dimensional null space, consisting of the constant functions.  Therefore, the inverse operator, $\Delta^{-1}$ is uniquely defined by the following convention:  
\begin{convention}\label{conv.laplinv}
The operator $\Delta^{-1}$ is the  linear operator on $L^2(M, g)$ with  a one-dimensional  null-space consisting of the constant functions:
\begin{equation}\label{e.constantsinnulllaplinv}
\Delta^{-1} 1=0
\end{equation}
Also, the image of $\Delta^{-1}$ is contained in the orthogonal complement of the constants in $L^2(M, g)$:
\begin{equation}
\int \Delta^{-1} f \dif V = 0 \hspace{1in} f\in L^2(M)
\end{equation}
 \end{convention}
Under the assumptions of Convention \ref{conv.laplinv}, the parametrics 
 for the Green's function may be given uniquely.  We review 
 some relevant properties  of the
Green's function in the following Proposition:
\begin{prop}[See, for example, Aubin \cite{Aubin}]
\label{propgreensfcn} On a closed, compact surface, $(M,g)$, let
$G(x,y)=K(\Delta^{-1},x,y)$ denote  the Green's function for the
positive Laplacian operator.
\begin{enumerate}
\item $G(x,y)$ is the unique integral kernel that is symmetric in $x$ and $y$, smooth for $d(x,y)>0$, integrates to $0$,
 and satisfies
\begin{equation}\label{e.greensfndistreqn}
\Delta^{distr}K(\Delta^{-1}, x,y) \stackrel{\dif V(y)}{=}
\delta_x(y)-\frac{1}{Vol(M)}
\end{equation}
\item If $d(x, y)<\epsilon$ for some small $\epsilon>0$, then there  is a continuous symmetric
 function,
 $H(x,y)\in C^0(M\times M)\cap C^{\infty}(M\times M\setminus\{(x,x)\})$, such that
\begin{equation}\label{parametricsgf}
G(x, y)=K(\Delta^{-1}, x,y) = \frac{-1}{2\pi}\log d(x,y)+H(x,y)
\end{equation}
\end{enumerate}
\end{prop}
As a result of Proposition \ref{propgreensfcn}, the following
regularization of the Green's function is well-defined:
\begin{definition} \label{d.robinmass}On a closed, compact surface, $(M,g)$, the {\it Robin's
mass}\footnote{The terminology comes from potential theory, in
which the regularization of the Green's function for the Lapalcian
on a domain  in Euclidean space  has been studied extensively},
is denoted by $m (x)$ and is defined as:
\begin{equation}\label{e.defrobinmass}
m(x)=\lim_{y\rightarrow x} \left[ G(x,y)+\frac{1}{2\pi} \log
d(x,y) \right] \end{equation}
\end{definition}
\paragraph{Spectral Background}
The Green's function may also be considered in terms of its spectral representation, and we briefly review some of the relevant spectral background in order to connect the Robin's mass and the geometrical mass to an appropriate spectral invariant.

The eigenfunctions for the Laplacian, denoted by
$\{\phi_j\}_{j=0}^\infty$ are smooth functions that form an
orthonormal basis for $L^2(M)$.  The eigenvalues are bounded
below, with $0<\lambda_1\leq \lambda_2 \cdots$, and Weyl's law
describes the growth of the eigenvalues:
\begin{equation}\label{e.weylslaw}
\lim_{j\rightarrow \infty} \frac{\lambda_j}{j}=\frac{4\pi}{Vol(M)}
\end{equation}
Of course, the Green's function is the integral kernel for $\Delta^{-1}$, may be written in terms of the spectral data by way of the following sum (which converges in $L^2(M\times M)$, but does not converge pointwise for $x=y$, as indicated by the logarithmic singularity appearing in \eqref{parametricsgf}:
\begin{equation}\label{e.greensspectraldata}
\Delta^{-1} f (x)= \int_M G(x,y) f(y)\dif V(y), \text{ where }  G(x,y)=\sum_{\lambda_j\neq 0} \frac{1}{\lambda_j} \phi_j(x) \phi_j(y)
\end{equation}

  Now, in order to make  connections with the regularized trace precise, we introduce   the operator $\Delta^{-s}$ , which is  a linear operator on   $L^2(M, g)$.  A branch of the logarithm is fixed  so that $\lambda^s=e^{s\log \lambda }$, with $1^s=1$.  In order to be consistent with Convention \ref{conv.laplinv}, the operator $\Delta^{-s}$ is uniquely defined by the  following requirements: 
  \begin{equation}\label{def.deltas}
\Delta^{-s} \phi_j(x)=\lambda_j^{-s} \phi_j(x) \hspace{2pt} \text{ for }  j\geq 1; \hspace{10pt} \Delta^{-s} 1 = 0; \hspace{10pt} \int_M \Delta^{-s} f \dif V \hspace{2pt} \text{ for } f\in L^2(M, g) 
\end{equation}
From Weyl's law, we see that  $\Delta^{-s}$ is a Hilbert-Schmidt operator for $\Re s\geq 1$, and when $s=1$, the  integral kernel for $\Delta^{-1}$ is precisely the Green's function, $G(x,y)$.  The expression for the Green's function given in \eqref{e.greensspectraldata} combined with the asymptotics for the eigenvalues    reveals that  the operator $\Delta^{-1}$   is not trace-class.  However,  eigenvalue growth demonstrates that the operator  $\Delta^{-s}$   is trace-class as long as  $\Re s>1$.  The spectral Zeta function is defined to be the trace of $\Delta^{-s}$:
\begin{equation}\label{e.defzeta}
Z(s)\stackrel{def}{=}\mathrm{Trace}(\Delta^{-s})=\sum_{j=1}^{\infty} \lambda_j^{-s} \hspace{30pt} \Re s>1
\end{equation}
The zeta function has  a simple pole at $s=1$, and it is well known that the zeta function may be continued to a meromorphic function defined 
 on the complex plane. In
\cite{Morpurgo}, Morpurgo considers the following spectral
invariant, which is  a regularization of the trace of $\Delta^{-1}$:
\begin{definition} \label{d.regzeta} Assume $(M,g)$ is a closed compact surface with volume normalized to be $4\pi$, then the
{\it regularized Zeta function at 1} is defined as follows:
\begin{equation}\label{e.defregzeta}
\tilde{Z}_g(1)=\lim_{s\rightarrow 1}
\left(Z(s)-\frac{1}{s-1}\right)
\end{equation}
\end{definition}
The following proposition gives a precise connection between the Robin's mass and the regularized Zeta function, stating  that the Robin's mass is a density for the regularized Zeta function:
\begin{prop}\label{p.massisspecdensity} [Morpurgo \cite{Morpurgo_Duke}, S. \cite{steiner_duke}]
Assume $(M,g)$ is a closed compact surface normalized with volume $Vol(M)=4\pi$.  Then, the Robin's
mass $m(x)$ is a density for the spectral invariant
$\tilde{Z}(1)$, defined in Equation \ref{e.defregzeta} (here
$\gamma$ is the Euler gamma constant):
\begin{equation}\label{e.massisspecdensity}
\int_M  m (x) \dif V(x) -2\log 2+2\gamma = \tilde{Z}(1)
\end{equation}
\end{prop}

\paragraph{Formula for behavior under conformal change of metric}
Since we will be interested in the behavior of the regularized trace under a conformal change of metric on a surface of genus $1$, we record the following formula, which is analogous to the Ray-Singer-Polyakov formula for the log-determinant.
\begin{theorem}[Morpurgo \cite{Morpurgo_Duke}, Steiner \cite{steiner_duke}]\label{thm.changeofzeta}  
If $\tilde{Z}(1)$ is the regularized trace for the constant curvature metric on a torus, with $\dif V$ and $\Delta^{-1}$ corresponding to the original metric $g$, then under conformal change $g_\phi = e^{2\phi} g$ fixing area to be $1$ on original metric and conformally changed metric.:
\begin{equation}\label{e.changeofzeta}
\tilde{Z}_\phi(1)-\tilde{Z}(1)=\frac{1}{2\pi}\int \phi(x) e^{2\phi} \dif V(x) -  \int \Delta^{-1} \etwophi
(x) e^{2\phi} \dif V(x)
\end{equation}
For any surface, the Robin's mass, $m(x)$ is the density for the regularized trace, and under the same conformal change of metric it satisfies the following formula:
\begin{equation}\label{e.changerobinmass}
m_{\!\!\:\phi}\!\!\;(x)=m(x)+ \frac{1}{2\pi}\phi(x)-2 \Delta^{-1} \etwophi
(x)+  \int_M \etwophi \Delta^{-1} \etwophi  \dif V
\end{equation}
\end{theorem}

\subsection{Connection between determinant and regularized trace}
In this paper, we will be concerned with the behavior of the regularized trace under conformal changes of metric, however, we should first make a few comments about the behavior of the regularized trace across conformal classes for flat tori.  The following theorem, which is given by Chiu in \cite{Chiu} states that the behavior of the regularized trace across conformal classes is linearly related to that of the log-determinant for the Laplacian.  
The theorem follows from the fact that the zeta function satisfies a  functional equation relating $Z(s)$ and $Z(1-s)$, so by considering the Laurent expansion of $Z(s)$ in the neighborhood of the pole at $1$, one can regularize and  finds that the residue at $1$ involves (the analytically continued) $Z'(0)$, which,  by definition, is  $\log\det(\Delta)=-Z'(0)$.
\begin{theorem}[Chiu, Theorem 2.3]  On a flat torus, $(T^2, g)$, the regularized trace for the Laplacian, $\tilde{Z}(1)$ defined in \eqref{e.defregzeta}, and the log-determinant for the Laplacian, $\log(\det(\Delta))=-Z'(0)$ are linearly related as follows. (Here $\gamma \approx .577216$ is the Euler Gamma constant) 
\begin{equation}\label{e.tracedet}
\tilde{Z}(1)=-\frac{\log(\det(\Delta)}{4\pi} -\frac{2\log 2}{2\pi}-\frac{1}{2\pi} \log \pi +\frac{\gamma}{2\pi}
\end{equation}
\end{theorem}
Since the extremal behavior of $\log(\det(\Delta))$ is understood across conformal classes, we can now describe the behavior of the regularized trace across conformal classes.
\begin{cor}[Chiu \cite{Chiu}, Osgood-Phillips-Sarnak \cite{Ops}]
If  $z_\tau=x+i y$ to be the representative for the area $1$  lattice generated by normalizing the basis $< (1, 0), \tau>$, and $\eta(z)$ is the Dedekind Eta function, a modular form of weight $\frac{1}{2}$, then the regularized trace is given as follows:
\begin{equation}\label{e.regtraceineta}
\tilde{Z}(1)=\frac{1}{4\pi}(-\log((2\pi)^2 y |\eta(z)|^4))+\frac{2\log(2\pi)}{4\pi} -\frac{2\log 2}{2\pi}-\frac{1}{2\pi} \log \pi +\frac{\gamma}{2\pi}
\end{equation}
 Also, among all flat-tori, the regularized trace, $\tilde{Z}(1)$ is minimized at the hexagonal lattice, where  $\tilde{Z}(1) \approx -.2286$.  \\
\end{cor}
By analyzing the asymptotic behavior of the Dedekind Eta function, one finds:
\begin{cor}
If the lattice corresponds to a rectangle with $\tau=\frac{a}{\pi}$, then for large $a$, the leading behavior of the regularized trace is $\tilde{Z}(1)\sim \frac{1}{12\pi} a$.  
\end{cor}
\begin{rmk}
 In \cite{hideandseek}, Doyle and the author give a probabilistic interpretation of the regularized trace as the ($\epsilon$-regularized) expected duration of a hide and seek game.  In the hide and seek game, a  seeker begins  at fixed point $x\in T^2$, and reach an $\epsilon$-neighborhood of someone who has hidden at a randomly chosen location on $T^2$; the expected duration of this game is essentially $\tilde{Z}(1)+\frac{1}{2\pi}\log\epsilon$.  The probabilistic interpretation gives the physical intuition for why, as  $a$ becomes large and the torus becomes longer and skinnier, the regularized trace should be growing: As the torus becomes longer, there will be more area (and therefore more places where the hider will need to be found) which is farther away from a fixed starting point.  \end{rmk}

 \section{Fat tori and long-skinny tori}\label{sec.longskinny}
 In this section we motivate the distinction between fat and long-skinny tori, which are defined in Definition \ref{d.longskinny}
 \begin{prop}[Morpurgo, \cite{Morpurgo}]  If $(T^2, g)$ is a torus with area $1$ and with the smallest eigenvalue satisfying $\lambda_1>8\pi$, then the the constant curvature metric is a local minimum for the regularized zeta function, $\tilde{Z}(1)$  within the class of conformally equivalent metrics with area $1$.  
 \end{prop}
\begin{proof}
 The proposition follows from a simple variational argument.    If we write the conformal factor as $e^{2\phi}=1+\lambda \psi$, with $\int_{T^2} \psi \dif V=0$ in order to keep the volume fixed to be $1$, then differentiating  \eqref{e.changeofzeta}, yields 
 the first variation for the regularized zeta function for the conformally equivalent metric, $\tilde{Z}_{\lambda}(1)$:
 \be
 \frac{\partial}{\partial \lambda}\tilde{Z}_{\lambda}(1)=\frac{1}{4\pi} \int_{T^2} (\psi+\psi \log(1+\lambda \psi))\dif V -\int_{T^2} 2 (1+\lambda \psi) \Delta^{-1} \psi \dif V
 \ee
 When $\lambda=0$, we have a critical point since $\psi$ integrates to $0$, and we the operator range of the operator $\Delta^{-1}$  is orthogonal to the constant functions, so 
 \be
 \frac{\partial}{\partial \lambda}|_{\lambda=0}\tilde{Z}_{\lambda}(1)=\frac{1}{4\pi} \int \psi \dif V+\int \Delta^{-1} \psi \dif V=0.
 \ee
 In order to understand the behavior at the critical point, we must consider the second variation:
 \be
   \frac{\partial^2}{\partial \lambda^2 }|_{\lambda=0}\tilde{Z}_{\lambda}(1)=\left[ \frac{1}{4\pi} \int_{T^2} \frac{\psi \cdot \psi}{ (1+\lambda \psi)}\dif V -\int_{T^2} 2  \psi) \Delta^{-1} \psi \dif V\right]_{\lambda=0}=\frac{1}{4\pi} \int_{T^2} \psi^2 \dif V -\int_{T^2} 2  \psi \Delta^{-1} \psi \dif V
 \ee
 Now, since $\psi$ is orthogonal to the null-space of $\Delta^{-1}$, and the non-zero eigenvalues for $\Delta^{-1}$ satisfy $\frac{1}{\lambda_j}\leq \frac{1}{\lambda_1}$, we have $\int \psi \Delta^{-1} \psi \dif V\leq \frac{1}{\lambda_1} \int \psi^2 \dif V$, so if $\lambda_1> 8\pi$, we see that the critical point is a minimum, since the second derivative is positive:
 \be
   \frac{\partial^2}{\partial \lambda^2 }|_{\lambda=0}\tilde{Z}_{\lambda}(1)=\frac{1}{4\pi} \int_{T^2} \psi^2 \dif V -\int_{T^2} 2  \psi \Delta^{-1} \psi \dif V\geq (\frac{1}{4\pi} -\frac{2}{\lambda_1} )\int_{T^2} \psi^2 \dif V\geq 0.
   \ee
   \end{proof}
   
At present, it is not known whether the minimum at the constant curvature metrics is actually a global minimum among all conformal metrics, however, we are interested in the case where $\lambda_1\leq 8\pi$, in which  a simple variational consideration of the formula for the change of the regularized trace under a conformal change does not yield any obvious behavior.   We introduce a more suggestive description of the types of tori in which we are interested.

 \begin{definition}\label{d.longskinny}  A torus $(T^2, g_{T^2})$, with flat metric $g_{T^2}$ is called a  {\it long-skinny torus}  if $\lambda_1$, the lowest non-zero eigenvalue for the Laplacian satisfies $\lambda_1>8\pi$.
 \end{definition}
 \begin{rmk}  The terminology `long-skinny' is readily apparent when one considers a rectangular torus.  In a rectangular lattice, since the volume of the torus must be fixed at $1$, the lattice will consist of a vector of length $w$ and another vector of length $\frac{1}{w}$.  The vectors and corresponding co-vectors are related to each other through the action of the element in $SL(2, \mathbb{Z})$ corresponding to the mapping $S: z\rightarrow \frac{-1}{z}$, and so the co-vectors will have lengths $w$ and $\frac{1}{w}$ as well.  Now,  $\lambda_1$ is the square  of the length of the shortest vector in the dual lattice,  so we may assume $w^2=\lambda_1$, and consequently if $\lambda_1>8\pi$, then $w>\sqrt{8\pi} $, and $\frac{1}{w}<\sqrt{8\pi}$, so there will be a `long' primitive closed geodesic   in the direction of the vector of length $w>\sqrt{8\pi}$, and a `skinny' primitive closed geodesic  in the direction of the vector with length $\frac{1}{w}<\sqrt{8\pi}$.
  \end{rmk}
  \begin{rmk}  Of course, the opposite of a long-skinny torus would be a fat torus, which has $\lambda_1>8\pi$, and from the previous proposition, we know that there is a local minimum for the regularized trace at the flat metric for fat tori.  .  The square torus and the torus corresponding to the hexagonal lattice are two well-known examples of fat tori..  \end{rmk}

 \section{Blowing bubbles in a skinny torus}\label{sec.blowbubbles}
 The results that we discuss here follow from a particular family of conformal factors on the flat torus, so we begin by describing the conformal factors.   The  picture in Figure \ref{fig.torus} of the Introduction  captures the basic behavior of the $C^{0,1}$ conformal factor,  which is obtained by mapping the torus to most of the  sphere. 
 
  \begin{theorem}\label{thm.blowbubble}  For a long-skinny rectangular torus of dimension $\sqrt{\frac{\pi}{a}}\times \sqrt{\frac{a}{\pi}}$ with flat metric, $g_{T^2}$, and area $1$, there exists a smooth conformally equivalent metric   $g_\phi=e^{2\phi}g$ and a corresponding regularized trace $\tilde{Z}_{\phi}(1)$ which approaches the value of the regularized trace on the standard sphere with area $1$, $\tilde{Z}_{S^2}(1)=\frac{1}{4\pi} (2\gamma-1-\log (4\pi))\approx -.189$    
  \be
  \tilde{Z}_{\phi} (1)=\tilde{Z}_{S^2}(1)+O(\frac{1}{a})
  \ee
  The conformal change reduces the regularized trace corresponding to the original flat metric, which for large $a$ has  the asymptotic behavior  $\tilde{Z}_{T^2}(1)=\frac{a}{12\pi}+o(a)$.  
 \end{theorem}
 \begin{rmk}
 Note that our result only holds for long-skinny tori.  For the torus corresponding to the hexagonal lattice, which is the global minimizer for the regularized trace among all flat tori, we have $\tilde{Z}_{hex}(1)\approx -.2286$, which is  smaller than that of the standard sphere with volume $1$, which is $\tilde{Z}(1)=\frac{1}{4\pi} (2\gamma-1-\log (4\pi))\approx -.189$.  
 \end{rmk}
 \begin{proof}
 Our method of proof is unabashedly constructive:  we exhibit an explicit conformal factor, and since the conformal factor is only $C^{0,1}$, we must show that one can find a smooth conformal factor for which the regularized zeta function will be as close as we like to our computed value.  Work is underway to give a more general argument that does not require the construction of an explicit conformal factor.   We now describe the conformal factor,  illustrated in Figure \ref{fig.torus}. The conformal factor  'blows a bubble' in the torus, because it is obtained from  mapping the torus to a sphere, and therefore making it behave like a spherical bubble.
 
 \paragraph{The  $C^{0,1}$ conformal factor}
 In order to make it slightly easier to state the mapping from the torus to the cap-less sphere, we work with a parametrization for the torus where we take $T^2=\{(x_1,x_2)\in [-a, a]\times [0, 2\pi]\}$ where the metric is a multiple of the Euclidean metric which will give area $1$, namely $g_{T^2}=\frac{1}{4\pi a} g_{Eucl}$.  
 The conformal factor depends only on the $x_1$ coordinate, since the symmetry in the $x_2$ coordinate will correspond to the rotational symmetry of the cap-less sphere, and we obtain the following conformal factor:
 \begin{equation}\label{e.c01conformal}
 g_{\phi}|_{(x_1, x_2)}=e^{2\phi}g_{T^2}|_{(x_1, x_2)}, \text{  with  } e^{2\phi}(x_1, x_2)=  \frac{a}{\tanh(a)\cosh(x_1)^2} 
 \end{equation}   
 Notice that the conformal factor is a nice smooth, even function of $x_1$,  as long as $x_1\in (-a,a)$, however, it is only Lipshitz  at $x_1=a$, which is identified with $x_1=-a$.
\\
Before proceeding to use this conformal factor, let's make a few more observations.
\begin{rmk}
 One can easily check that the conformal factor will have area $1$, since:
\begin{eqnarray*}
\mathrm{Area}(\dif V_{\phi})
&=&
\iint\limits_{[(0,2\pi)\times[-a,a]} \frac{a}{\tanh(a)} \frac{1}{\cosh(x_1)^2}\dif V _{T^2}
\\&=&
\frac{a}{\tanh(a)} 2\pi \frac{1}{4\pi a} \int_{-a}^a\frac{1}{\cosh(x_1)^2} \dif x_1
\\&=&\frac{2\tanh(a)}{2 \tanh(a)}
\\&=&
1
.
\end{eqnarray*}
\end{rmk} 
 \begin{rmk}  As mentioned before, the conformal factor in \eqref{e.c01conformal} was obtained by mapping the  long, skinny torus to a sphere, and then pulling back the metric and re-scaling as necessary.  Since we may think of the torus as a cylinder with ends identified, the conformal factor comes from the well-known conformal diffeomorphism of an infinite cylinder onto the sphere without the two poles.  In particular, the map $\Psi$ from the rectangle $[-a,a]\times [0, 2\pi]$  to the sphere is given as follows:
 \be
\Psi(x_1, x_2)=\left(\frac{2 e^{x_1}}{1+e^{2x_1}} \cos(x_2), \frac{2 e^{x_1}}{1+e^{2x_1}} \sin(x_2), \frac{1-e^{2x_1}}{1+e^{2x_1}} \right) \in S^2
\ee
One can compute the pulled back metric is conformally equivalent to the standard Euclidean metric on the plane with coordinates  $(x_1, x_2)$, with $\Psi^\ast g_{S^2}=\frac{1}{\cosh^2(x_1)} g_{Eucl}$, and since we are demanding that the original and conformal metrics both have area $1$, the pulled-back metric must be re-scaled by the factor of $4\pi \tanh(a)$, which is the area of the image on the sphere of the mapping of the rectangle by $\Psi$.    
\end{rmk}
In order to verify the fact that this $C^{0,1}$ conformal factor has a regularized trace, $\tilde{Z}_\phi(1)$ which is close to that of the sphere, we enlist the help of Mathematica in order to compute the integrals that arise.     We computed the change in the regularized zeta functions of the conformal and original metrics, $\tilde{Z}_\phi(1)-\tilde{Z}(1)$ which is given explicitly in terms of the data for the flat metric, the conformal factor, and $\Delta^{-1} e^{2\phi}$ in Equation \ref{e.changeofzeta}.  In order to get our hands on an explicit formula for $\Delta^{-1} e^{2\phi}$, we apply a bit of trickery which is described below.  After computing the functional for the change of the regularized trace,   in order to isolate $\tilde{Z}_\phi(1)$, we use the  explicit expression for  $\tilde{Z}(1)$ given in Equation \ref{e.regtraceineta} in terms of the Dedekind Eta function (which Mathematica knows).  

Before delving more closely into the nitty-gritty details of our computation, we must observe that our discussion of the spectral theory  for the Laplacian and the definition of $\tilde{Z}(1)$ has taken place entirely within the class of $C^\infty$ metrics.  Since   the $C^{0,1}$ conformal factor given in equation \ref{e.c01conformal} takes us outside of the class of smooth metrics, we have not actually defined what would be meant by $\tilde{Z}_(1)$ for such a Lipshitz metric.  
Rather than considering  spectral theory for  Laplacians with rough coefficients,  we will skirt the issue by defining the quantity as follows:  
\begin{equation}\label{e.defregzetanotsmooth}
\tilde{Z}^0_\phi(1)\stackrel{def}{=}\left( \frac{1}{2\pi}\int \phi(x) e^{2\phi} \dif V(x) -  \int \Delta^{-1} \etwophi
(x) e^{2\phi} \dif V(x)\right)+ \tilde{Z}(1) \text{  for   } \phi\in C^{0}
\end{equation}
Of course, the expression for the change of the regularized trace in \eqref{e.changeofzeta} shows that this definition agrees with the usual defintion of the regularized trace if $\phi\in C^{\infty}$, and the next Proposition shows that one can find a sequence of smooth conformal factors, $e^{2\phi_\epsilon}$ for which the values of $\tilde{Z}_{\phi_\epsilon}(1)$ will approach the number $\tilde{Z}^0_\phi(1)$.
  \begin{prop}  Assume that $(T^2, g)$ is the smooth flat metric, and $\tilde{Z}(1)$, $\Delta^{-1}$ and $\dif V$ correspond to the smooth background metric.  For any $\phi\in C^0(M)$ , for any $\epsilon>0$, there is a smooth  function $\phi_\epsilon$, for which  the functional  $F[\phi]=\frac{1}{2\pi}\int \phi(x) e^{2\phi} \dif V(x) -  \int \Delta^{-1} \etwophi
(x) e^{2\phi} \dif V(x)
$, may be approximated with  $|F[\phi_\epsilon]-F[\phi]|< \epsilon$.  Consequently, with $\tilde{Z}^0_\phi(1)$ defined for the non-smooth conformal factor in \eqref{e.defregzetanotsmooth}, we have:
\be
|\tilde{Z}^0_\phi(1)-\tilde{Z}_{\phi_\epsilon}(1)|=|F[\phi_\epsilon]-F[\phi]|<\epsilon \text{ or  } \tilde{Z}_{\phi_\epsilon}(1)=F[\phi]+\tilde{Z}(1)+O(\epsilon)\
\ee
\end{prop}
\begin{proof}
 We first note that $F[\phi]$ is well-defined and finite for any $C^0$ function.  Since $M$ is compact, the first integral is clearly well-defined and finite.  The finiteness of the second integral follows easily from elliptic regularity and the  Sobolev Imbedding Theorem: by definition, $\Delta^{-1}(e^{2\phi})$ satisfies $\Delta(\Delta^{-1}(e^{2\phi})=e^{2\phi}-1$, where $e^{2\phi}-1$ will  be in $C^0$, and if $\phi\in C^0$,   by elliptic regularity, we know that $\Delta^{-1} e^{2\phi}$ is in the Sobolev space $H_2^2$ (two derivatives in $L^2$), and finally,   the Sobolev theorem says that  $H_2^2 \subset C^{0,\alpha}$ for dimension $n=2$ as long as $\alpha < 1$, so $\Delta^{-1} e^{2\phi} \in C^{0, \alpha}$, and therefore   $\int (\Delta^{-1} e^{2\phi} ) e^{2\phi}$ is finite.  Moreover, by considering the fact that $\Delta^{-1}e^{2\phi}$ has an integral kernel representation (the integral kernel is, of course, the Green's function), so we know that $\int (\Delta^{-1} e^{2\phi}) e^{2\phi}$ is continuous with respect to $\phi$, and thus  $F$ is a continuous operator on $C^0(M)$.  Now, since smooth functions form a  dense subset of the continuous functions, there is some $\phi_\epsilon$ such that we can approximate $F[\phi]$ with $F[\phi_\epsilon]$ as desired.
\end{proof}
Now we can return to our describing our computation for the $C^{0,1}$ metric with a clear conscience.

\paragraph{Computing $F[\phi]=\frac{1}{2\pi}\int \phi(x) e^{2\phi} \dif V(x) -  \int \Delta^{-1} \etwophi
(x) e^{2\phi} \dif V(x)$}
     Before setting Mathematica to the task of computing the functional $F[\phi]$, we need to compute $\Delta^{-1} e^{2\phi}$.   According to our definition of the operator $\Delta^{-1}$, the potential must satisfy the following two requirements:
\be
\Delta (\Delta^{-1}(e^{2\phi})=e^{2\phi}-1 \text{   and   } \int_{T^2} (\Delta^{-1}(e^{2\phi}) \dif V_{T^2}=0
\ee
We compute the potential in three steps:
\ben
\item We first exploit the fact that the conformal factor arises from the map to the sphere, so for any point in the interior of the rectangular fundamental domain, the Gaussian curvature of the conformally equivalent metric will be exactly $1$.  From the change of curvature equation, we have (locally):
\be
K_\phi=1=\frac{1}{e^{2\phi}} (\Delta \phi +K_{T^2})=\frac{1}{e^{2\phi}} \Delta \phi  \text{  or } e^{2\phi}=\Delta \phi
\ee
Consequently, we see that $\phi$ almost solves the equation $\Delta (\Delta^{-1}(e^{2\phi})=e^{2\phi}-1$.  
\item  Now, from the curvature equation we see that  we must have $\Delta^{-1}e^{2\phi}=\phi+f$, where $\Delta f = -1$.  The need for the corrective function $f$ should not be  surprising given the local nature of the change of curvature equation, and the necessarily global behavior of a solution to a Poisson equation.  Since the metric on $T_2$ is just a scaled version of the Euclidean metric, and we know that $\Delta_{Eucl} (x_1)^2=-(\partial_{x_1, x_1}(x_1)^2+\partial_{x_2, x_2}(x_1)^2)=-2$, and we  find that the correct scaling is given by  $f=\frac{1}{8\pi a} (x_1)^2+c$.  Up to the normalization of the additive constant, we have
\be 
\Delta^{-1} e^{2\phi}=\phi+\frac{1}{8\pi a} (x_1)^2+c \text{  satsifies  } \Delta (\Delta^{-1}(e^{2\phi})=e^{2\phi}-1
\ee
\item  Finally, the constant $c$ is chosen to ensure that $\int \Delta^{-1} e^{2\phi} \dif V = 0$, so we have \be 
c=\frac{-1}{2a} \int_{-a}^a \left(  \frac{1}{2}\log(\frac{a}{\tanh(a)\cosh(x_1)^2} )   + \frac{(x_1)^2}{8\pi a} \right)\dif x_1, 
\ee
and consequently, the potential for the conformal factor,  $\Delta^{-1} e^{2\phi}$ is given by:
\be
\Delta^{-1} e^{2\phi}(x)=\phi(x)+\frac{1}{8\pi a} (x_1)^2+ \frac{-1}{2a} \int_{-a}^a \left(  \frac{1}{2}\log(\frac{a}{\tanh(a)\cosh(x_1)^2} )   + \frac{(x_1)^2}{8\pi a}\right) \dif x_1
\ee
\een 
 We then use Mathematica to compute the expression for the change of the regularized Zeta function for the original flat metric and the conformally equivalent metric, and the expression exhibited the correct asymptotics, with $\tilde{Z}_\phi-\tilde{Z}(1)=\frac{-a}{12\pi}+o(a)$, which easily demonstrates that the conformal change of metric reduces the value of the regularized trace.\footnote{According to Mathematica, $F[\phi]=\frac{1}{48 a (-1+e^{4 a}) \pi }
(\text{Coth}[a] (\text{Cosh}[2 a]+\text{Sinh}[2 a]) (48 a-64 a^2-48 a \text{Cosh}[2 a]+
16 a^2 \text{Cosh}[2 a]-24 a \text{Log}[4]+24 a \text{Cosh}[2 a] \text{Log}[4]+
48 a \text{Log}[1+e^{2 a}]-48 a \text{Cosh}[2 a] \text{Log}[1+e^{2 a}]-
24 a \text{Log}[\frac{a (1+e^{2 a})}{-1+e^{2 a}}]+24 a \text{Cosh}[2 a] \text{Log}[\frac{a (1+e^{2 a})}{-1+e^{2 a}}]+
24 a \text{Sinh}[2 a]+12 a^2 \text{Sinh}[2 a]+\pi ^2 \text{Sinh}[2 a]+
48 a \text{Log}[1+e^{2 a}] \text{Sinh}[2 a]-6 \text{PolyLog}[2,-e^{-2 a}] \text{Sinh}[2 a]+
18 \text{PolyLog}[2,-e^{2 a}] \text{Sinh}[2 a]))$.}:
 \be
 F[\phi]=\tilde{Z}^0_\phi(1)-\tilde{Z}(1)=\frac{1}{2\pi}\int \phi(x) e^{2\phi} \dif V(x) -  \int \Delta^{-1} \etwophi
(x) e^{2\phi} \dif V(x)
\ee
\paragraph{Computing $\tilde{Z}^0_\phi(1)$}
In Section \ref{sec.background} we had recorded Chiu's observation that  the regularized Zeta function $\tilde{Z}(1)$  for the flat torus is linearly related to the log-determinant, and therefore, by using the expression in \eqref{e.regtraceineta}, $\tilde{Z}(1)$    may be written in terms of  the Dedekind Eta function,  which Mathematica can compute. Mathematica then computed: 
\begin{eqnarray*}
\tilde{Z}^0_\phi(1)
&=&
\tilde{Z}(1)+ \frac{1}{2\pi}\int \phi(x) e^{2\phi} \dif V(x) -  \int \Delta^{-1} \etwophi
(x) e^{2\phi} \dif V(x)
\\&=&
\frac{1}{4\pi}(-\log((2\pi)^2 \frac{\pi}{a}  |\eta(\frac{i \pi}{a})|^4))+\frac{2\log(2\pi)}{4\pi} -\frac{2\log 2}{2\pi}-\frac{1}{2\pi} \log \pi +\frac{\gamma}{2\pi}+
\\&&
\frac{1}{2\pi}\int \phi(x) e^{2\phi} \dif V(x) -  \int \Delta^{-1} \etwophi
(x) e^{2\phi} \dif V(x)
\end{eqnarray*}
The plot of this expression against $a$ shows that as $a$ grows and therefore the rectangle becomes longer and skinnier $\tilde{Z}_\phi(1)$ approaches the value on the unit area sphere $\tilde{Z}_{S^2}(1)\approx -.189$.    \\
\begin{figure}[h]
\hspace{1.5in}\includegraphics[width=2in]{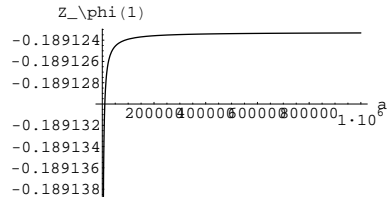}
\caption{ As $a$, the parameter of long-skinniness grows, $\tilde{Z}_\phi$ approaches $\tilde{Z}_{S^2}(1) \approx -.189$ }
\label{f.mathematicagraph}
\end{figure}
Now, if one does not like the fact that we are using the numerical behavior in order to verify the Theorem \ref{thm.blowbubble}, we should point out that all of the integrals and functions are explicitly known, so one could do an asymptotic analysis of all of the integrals involved in order to show that $\tilde{Z}_\phi =\tilde{Z}_{S^2}(1)+O(\frac{1}{a})$ directly. 
\paragraph{An alternate approach: smoothing out the conformal factors}
We should also comment on an alternate method of proof of Theorem \ref{thm.blowbubble}.    At present, we will only outline the method, because it is currently tailored to this particular set of conformal factors, and we are in process of developing a more generalizable method of making these kinds of estimates.  One could work with a smoothed out version of the conformal factor given in \eqref{e.c01conformal}.  If we take a bump function, $\psi(x)$  that is $1$ on $[-a+1, a-1]$, and $0$ on $[-a, -a+1 \frac{1}{4\sqrt{\pi}}]\cup [a-1+\frac{1}{4\sqrt{\pi}}, a] $.   The smoothed-out conformal factor may be given as follows:
\begin{equation}
e^{2\tilde{\phi}}=\left\{
\begin{array}{r@{\quad\quad\quad} l}
\frac{a}{\cosh^2(x_1)} & x_1\in [-a+1, a-1] \text{  (bubble)} \label{e.conffactor}\\
\frac{a}{\cosh^2(x_1)}\psi(x)+c_2(1-\psi(x))  & x_1\in [-a+1-\frac{1}{4\sqrt{\pi}}, -a+1]\cup [a-1, a-1+\frac{1}{4\sqrt{\pi}}] \nonumber \\
c_2=O(e^{-2a})   & x_1\in [-a, -a+1-\frac{1}{4\sqrt{\pi}}]\cup [a-1+\frac{1}{4\sqrt{\pi}}, a] \nonumber
\end{array}
\right.
\end{equation}
The area of the`bubble' region on which $x_1 \in [-a+1, a-1]$ will be $\tanh(a-1)=1+O(e^{-2a})$, and the choice of endpoints is made with the injectivity radius in mind.  One can then give a proof of Theorem \ref{thm.blowbubble} by writing $\tilde{Z}_{\tilde{\phi}}(1)$ in terms of the integral of the Robin's mass, $m_{\tilde{\phi}}(x)$, which is the mass-like term in the Green's function for the Laplacian.   The relationship between the regularized trace and the integral of the Robin's mass was given by Morpurgo in \cite{Morpurgo_Duke}, and by Steiner in \cite{steiner_duke}, and is given as follows for the data of any metric $g$ (as was recorded in Proposition \ref{p.massisspecdensity}:
 \be
\tilde{Z}_(1)=\int m(x) \dif V(x) -2\log 2 + 2\gamma  \text{   where   } m(x)=\lim_{y\rightarrow x}[ G(x,y)+\frac{1}{2\pi} \log (d(x,y))]
\ee
If we work with this formulation of the regularized trace, we can prove Theorem \ref{thm.blowbubble} by showing that on the bubbled region, the Green's function for the conformally equivalent metric is close to that of the Green's function for the sphere, and therefore $m_{\tilde{\phi}}(x)\approx m_{S^2}$ for most of the region, and on the remaining  regions in the torus (there is a region on which  the metric is constant and $O(e^{-2a})$, and there is a  gluing region connecting the bubble to the flat part), the contribution of the integral of the Robin's mass is suitably small.  To be more precise, we can prove the following two statements:
we will show that for $x\in T^2$, where $x$ is at the edge of or outside of the spherical portion,  the contribution of $\int m_\phi(x) \dif V_\phi (x)$ to $\tilde{Z}(1)$ is small, and for $x\in T^2$, in the spherical portion, we will show that $m_\phi(x) \approx m_{S^2}$, so for most of the area of the torus, we will have $\int m_\phi(x) \dif V_\phi(x)=\int m_S^2 \dif V_\phi(x)$.    Now, we will say a few more words about how these two steps are carried out.
\ben
\item  In order to see that $m_\phi(x)\approx m_{S^2}$ for $x$ in the bubble region, recall that the map $\Psi$ is a diffeomorphism from the torus to the sphere, and we use $\Psi$ to pull-back the spherical metric in order to obtain the conformal factor on the bubble region.  If $x$ is well-within the  bubble region, \ie $x=(x_1, x_2)$ with $x_1\in (-a+2, a-2)$, and $\Phi$ is a bump function that is supported within the bubble region, then  for any $y\in T^2$, the Green's function is given by $G_{\tilde{\phi}}(x,y)=G^\ast_{S^2}(\Psi(x),\Psi(y)) \Phi(x)+F(x,y)$.  We can easily show that for any $x$, the $C^0$ norm of $F(x,y)$ is small, by using an argument that is similar to the argument one gives in constructing the Green's function.  Consequently, $m_{\tilde{\phi}}(x)\approx m_{S^2}+||F(x,\cdot)||_{C^0}$ on a region of the torus that has most of the volume of the conformally equaivalent metric, since the bubble region has  volume $1-O(e^{-2a})$.  

\item If $x=(x_1, x_2)$ with $x_1 \in [-a, -a+2]\cup [a-2, a]$, then we can show that, at worst,  $|m_{\tilde{\phi}}|\leq p(a)$, where $p(a)$ is a polynomial in $a$.  However, since the area of this region is $O(e^{-2a})$, by a simple  $0; \infty$  H\"{o}lder estimate, we see that the contribution of the integral of the Robin's mass on this region will be small.
\een
At present, we do not give further details on these estimates, because they have been specifically suited to this particular example.  Work is underway to develop more broadly applicable methods of estimating the Green's function.  This is a subtle business, because one can show that there is a bubbling-off type phenomenon that can drive the regularized trace up.  
\end{proof}

\section*{Addendum: Any torus can beat the sphere}
Using conformal factors which, like ours, depend only on longitude,
Okikiolu
\cite{Okikiolububble}\label{pushdown}
has shown that
within the conformal class of any torus
there is a metric with trace less than that of the sphere.
Looking back at 
Figure \ref{f.mathematicagraph},
we observe that the approach of the trace of our bubbles
to that of the sphere is \emph{from below}.
In hindsight, we should have expected this.
The trace measures how long it takes a Brownian particle
to get near a randomly selected point.
(See Doyle and Steiner \cite{hideandseek}.)
Our bubbled tori are like a sphere with a portal that teleports
you from near one pole to near the other.
This shortcut should make it easier for a Brownian particle to
get around on the surface and find a randomly selected target point.

Now, the graph in Figure \ref{f.mathematicagraph}
pertains to rectangular tori.
But adding a twist before glueing up
a skinny cylinder to a flat torus will only decrease the trace.
This can be verified by computation, but in a second we'll give
a probabilistic explanation.
So, twisting before glueing up a skinny torus decreases its trace.
But now the further decrease of the trace that we get from
our longitudinal bubbling
is exactly the same with or without the twist,
because the twisting doesn't show up when you use
Morpurgo's formula \eqref{e.changeofzeta}
to compute the change of the trace under a longitudinal conformal variation.
So longitudinal bubbling can push the trace of any sufficiently skinny torus,
rectangular or not,
below that of the sphere.

The reason twisting before glueing should decrease the trace of a skinny torus
is that on a rectangular flat torus,
the longitudinal diffusion of a Brownian particle
is independent of its meridional diffusion.
Putting in a twist destroys this independence
by giving a particle that has drifted all the way
around longitudinally a shove around the meridian.
This ought to help it get near a randomly selected target point.
This is all providing the torus was sufficiently skinny:
This argument is not going to hold up for a fat torus,
or a skinny torus where you have misidentified the longitudinal direction.
The effect will be slight, because by the time it gets all the way
around a skinny torus the long way a Brownian particle
will have nearly forgotten where it started meridionally.
Still, there will be some effect, and it should work in the right
direction.
Once again, this can be verified by computation.

Now, as indicated, this all pertains only to sufficiently skinny tori.
But flat tori that are sufficiently fat start with trace less than
that of a sphere.
And you can check
(cf. Okikiolu
\cite{Okikiolububble})
that any torus is either sufficiently skinny for
our method to work, or sufficiently fat for our method to be unnecessary.
So the trace of any torus can be made less than that of a sphere.

We only realized all this after seeing
Okikiolu
\cite{Okikiolububble}.
Her work is based on identifying the longitudinal conformal factor that
most decreases the trace, and thus goes far
beyond simply producing one suitable conformal factor,
as we have done here.

We originally developed
this longitudinal bubble approach because we didn't know how to
blow a bubble locally, within any small patch on the torus.
Or rather, we didn't know how to show that the result has 
trace close to that of a sphere.
Okikiolu
\cite{Okikiolutrace}
solves this problem.
It's a relief to know that you can blow bubbles locally,
and get the trace close to that of a sphere.
Meanwhile, it's gratifying that our longitudinal bubbles,
which we viewed as a poor substitute for local bubbles,
actually turn out to work better, and
push the trace below that of the sphere.

We close by noting that we still
lack tools to translate probabilistic
intuitions about the trace into proofs.
Good examples are the plausibility arguments above purporting
to explain why the trace decreases when
you add a small portal to teleport you between the poles of a sphere,
or twist before glueing up a skinny cylinder.
We have tried to explain these facts 
probabilistically,
but in the end we are reduced to verifying them by computation.
We would like tools that would let us
nail down these probabilistic
intuitions and translate them into proofs.
These tools needn't be probabilistic in nature.
In many cases, the most useful role of probability in analysis is advisory:
Probabilistic intuition suggests the way to an analytic proof.
Of course, we can always let probabilistic intuition suggest what might be
true, but we want more than that.
We want the probability to suggest not just the result, but the way to 
prove it.

\bibliographystyle{hplain}
\bibliography{bubble}

\end{document}